\documentclass{amsart}

\usepackage{amsthm,amsfonts,amsmath,amscd,amssymb,graphicx,tikz-cd,mathtools,mathrsfs,url,upgreek,comment}

\newtheorem{theorem}{Theorem}

\newtheorem{proposition}[theorem]{Proposition}

\newtheorem{definition}[theorem]{Definition}

\newtheorem{example}[theorem]{Example}

\numberwithin{equation}{section}

\newtheorem{conjecture}[theorem]{Conjecture}

\newtheorem{remark}[theorem]{Remark}

\renewcommand{\labelenumi}{(\roman{enumi})}

\def\R{{\mathbb R}}

\def\C{{\mathbb C}} 
\def\Z{{\mathbb Z}}

\def\gog{{\mathfrak g}} \def\tot{{\mathfrak t}}

\def\U{{\rm U}}

\begin{document}

\title{Majorization and Spherical Functions}


\author{Colin McSwiggen and Jonathan Novak}





\begin{abstract}
In this paper, we generalize a result of Cuttler, Greene, Skandera, and Sra that characterizes the majorization order
on Young diagrams in terms of nonnegative specializations of Schur polynomials. More precisely,
we introduce a generalized notion of majorization associated to an arbitrary crystallographic root system $\Phi,$ and show that it admits a natural characterization in terms of the values of spherical functions on any Riemannian symmetric space with restricted root system $\Phi.$  We also conjecture a further generalization of this theorem in terms of Heckman--Opdam hypergeometric functions.
\end{abstract}

\maketitle

\section{Introduction}
Given a positive integer $d$ and a pair of Young diagrams $\lambda,\mu \vdash d,$ we write $\lambda \succeq \mu$ if,
for any $r \geq 1,$ the number of cells in the first $r$ rows of $\lambda$ is at least as large as the number of cells in the 
first $r$ rows of $\mu.$ This relation --- which is known as \emph{majorization} --- defines a partial order on the set of Young diagrams with
$d$ cells. Given a Young diagram $\lambda$ and a positive integer $N,$
a Young tableau of shape $\lambda$ and rank $N$ is a function $T$ on the cells of $\lambda$ which takes values in
$\{1,\dots,N\}.$ We say that $T$ is \emph{semistandard} if it is weakly increasing along rows and strictly increasing along
columns. If $\lambda \vdash d \leq N,$ then the set $\mathrm{SSYT}(\lambda,N)$ of semistandard Young tableaux of shape $\lambda$
and rank $N$ is nonempty, and the \emph{normalized} Schur polynomials

	\begin{equation*}
		S_\lambda(x_1,\dots,x_N) = \frac{1}{|\mathrm{SSYT}(\lambda,N)|} \sum_{T \in \mathrm{SSYT}(\lambda,N)} 
		\prod_{j=1}^N x_j^{|T^{-1}(j)|}, \qquad \lambda \vdash d,
	\end{equation*}

\noindent
form a basis of the space of homogeneous symmetric polynomials of degree $d$ in $N$ variables. 
The starting point of this paper is
a result of Cuttler--Greene--Skandera \cite{CGS} and Sra \cite{Sra},\footnote{Ait-Haddou and Mazure have also given a different proof of this theorem via the theory of blossoms \cite{AH-M}.} which states that the polynomials $S_\lambda$ characterize the majorization order in the following sense.
	
	\begin{theorem}
	\label{thm:CGSS}
	For any $\lambda,\mu \vdash d \leq N,$ we have $\lambda \succeq \mu$ if and only if
	
		\begin{equation*}
			S_\lambda(x_1,\dots,x_N) \geq S_\mu(x_1,\dots,x_N) \quad \textrm{for all } \, x_1,\dots,x_N \in \mathbb{R}_{\geq 0}.
		\end{equation*}
	\end{theorem}

Quite recently, Khare and Tao \cite{KhTao, KhTao2} have provided an analytic
extension of Theorem \ref{thm:CGSS} by showing that the functions $S_\lambda$ can also be used to 
characterize \emph{weak} majorization, i.e.~the preorder obtained by dropping the condition $|\lambda|=|\mu|$ in 
the definition of majorization. The present paper provides a geometric generalization of Theorem \ref{thm:CGSS}:
we consider a new notion of majorization associated to the Weyl group of a given root system\footnote{In this paper all root systems are assumed to be crystallographic, but not necessarily reduced.} $\Phi$, and show that it is characterized by a pointwise inequality 
for spherical functions on any Riemannian symmetric space with restricted root system $\Phi$. 

The paper is organized as follows.

In Section \ref{sec:majorization}, we introduce the notion of Weyl group majorization and prove our main result in the special 
case of spherical functions on the Lie algebra $\gog$ of a compact group $G$ (Theorem \ref{thm:H-W-convexity}).  We treat this case separately 
because spherical functions reduce to Harish-Chandra orbital integrals in this setting, leading to a very natural and direct 
generalization of Theorem \ref{thm:CGSS}. Moreover, the proofs in the Lie algebra case do not require a discussion of symmetric spaces, 
and are thus somewhat more elementary. 

In Section \ref{sec:symspaces}, we treat the general case of spherical functions on a Riemannian symmetric space of 
non-compact type, and prove our main result (Theorem \ref{thm:spherical-majorization}).  
We then use this theorem to deduce majorization-characterizing inequalities for spherical functions on symmetric spaces of 
Euclidean type (Proposition \ref{prop:euclidean-majorization}) and compact type (Proposition \ref{prop:compact-majorization}).
Furthermore, we discuss how the result for the compact case implies inequalities 
for various families of orthogonal polynomials, such as Schur polynomials.  

In Section \ref{sec:conjecture}, we develop an even more general framework based on 
Heckman--Opdam hypergeometric functions \cite{RSHF1}.  We state a conjectural characterization of majorization in this context,
and prove one direction of this conjecture. Moreover, we show that the full conjecture holds in rank one, where it reduces to an 
inequality for the classical Gauss hypergeometric function.

\section{Lie Algebras}
\label{sec:majorization}

Let $G$ be a connected, compact Lie group with Lie algebra $\gog$.  Let $T \subset G$ be a maximal torus, $\tot = \mathrm{Lie}(T) \subset \gog$ the corresponding Cartan subalgebra, and $\Phi$ the root system of $\gog$ with respect to $\tot$.  Fix a choice of positive roots $\Phi^+$, and let $W$ be the Weyl group generated by reflections in the root hyperplanes in $\tot$.

\begin{definition} \label{def:W-majorization} \normalfont
Let $V$ be any vector space on which $W$ acts by reflections, and let $\lambda, \mu \in V$.  We say that $\lambda$ $W$-{\it majorizes} $\mu$, written $\lambda \succeq \mu$, if $\mu$ lies in the convex hull of the Weyl orbit of $\lambda$.
A function $F: V \to \R$ is said to be $W$-{\it convex} if $F(\lambda) \ge F(\mu)$ whenever $\lambda \succeq \mu$. 
\end{definition}

The relation $\succeq$ defines a preorder on $\tot$ and a partial order on each 
Weyl chamber.\footnote{The $W$-majorization preorder is not the same as the height partial order on $\tot$.  If $\lambda, \mu \in \tot$, we say that $\lambda$ is {\it higher} than $\mu$ if $\lambda - \mu$ can be written 
as a linear combination of positive roots with nonnegative coefficients. The resulting order coincides with $W$-majorization when $\lambda$ and $\mu$ are both dominant, but the height partial order is not $W$-invariant.}  
When $G = \mathrm{U}(N),$ so that $W \cong S_N$ and $\tot \cong \R^N$, $W$-majorization coincides with the usual majorization order on vectors, and $W$-convexity coincides with the widely studied notion of Schur-convexity.  In particular, we recover the ordering on Young diagrams by regarding them as vectors with weakly decreasing integer coordinates.
Generalized majorization orders associated to group actions have been studied since the 1960's; see \cite{AOA,EP, GW, Mudh1}. 

In this section, we prove that $W$-majorization may be characterized by comparing the pointwise behavior of Laplace transforms of invariant probability measures on coadjoint orbits of $G$.  
Concretely, choose an invariant inner product $\langle \cdot,\cdot \rangle$ identifying $\gog \cong \gog^*$, and for $\lambda \in \gog$ let $\mathcal{O}_\lambda = \{ \mathrm{Ad}_g \lambda \ | \ g \in G \}$ denote its (co)adjoint orbit.  
Define
\begin{equation} \label{eqn:L-def}
L_{\lambda}(x) = \int_{\mathcal{O}_\lambda} e^{\langle y,x \rangle} dy = \int_G e^{\langle \mathrm{Ad}_g \lambda, x \rangle} \, dg, \qquad \lambda, x \in \tot_\C,
\end{equation}
where $dy$ is the unique invariant probability measure on $\mathcal{O}_\lambda$, $dg$ is the Haar probability measure on $G$, and $\tot_\C \cong \tot \oplus i \tot$ is a 
Cartan subalgebra of the complexified Lie algebra $\gog_\C$.

The functions $L_{\lambda}(x)$ are ubiquitous objects that arise in many different areas of mathematics and physics.  They were originally studied by Harish-Chandra in the context of harmonic analysis on Lie algebras \cite{HC}, 
and they play an important role in the orbit method in representation theory \cite{AK}.  When $G = \U(N)$, the transform $L_\lambda(x)$ is known 
as the Harish-Chandra--Itzykson--Zuber (HCIZ) integral; it has been widely studied in theoretical physics and random matrix theory since the 1980's \cite{IZ}.  More recently, the HCIZ integral
has become an important object in combinatorics \cite{GGN,Novak:New} and probability \cite{BGH,Novak:JSP}.
For further background on these functions and their diverse applications, we refer the reader to \cite{McS1, McS2}.

The main result of this section is the following characterization theorem.

\begin{theorem} \label{thm:H-W-convexity}
For any $\lambda, \mu \in \tot$, the following are equivalent:
\begin{enumerate}
\item $\lambda \succeq \mu$,
\item $L_\lambda(x) \ge L_\mu(x)$ for all $x \in \tot$.
\end{enumerate}
\end{theorem}

\begin{proof}
We first show (ii) implies (i), by proving the contrapositive.  The {\it discriminant} of $\gog$ is the polynomial $\Delta_\gog(x) = \prod_{\alpha \in \Phi^+} \langle \alpha, x \rangle$, i.e.~the product of the positive roots.\footnote{Here we take the roots to be real-valued linear functionals in $\tot^*$, which we identify with $\tot$ via the inner product.  As a result, our roots differ by a factor of $i$ from those typically used in the setting of complex semisimple Lie algebras.}  Let $x \in \tot$ with $\Delta_\gog(x) \not = 0$.  We assume for now that $\Delta_\gog(\lambda), \Delta_\gog(\mu) \not = 0$ as well; later we will remove this assumption.  The Laplace transform (\ref{eqn:L-def}) admits an exact expression, due to Harish-Chandra \cite{HC}:
\begin{equation} \label{eqn:hc}
L_\lambda(x) = \frac{\Delta_\gog(\rho)}{\Delta_\gog(\lambda) \Delta_\gog(x)} \sum_{w \in W} \epsilon(w) e^{\langle w(\lambda), x \rangle},
\end{equation}
where $\rho = \frac{1}{2} \sum_{\alpha \in \Phi^+} \alpha$ and $\epsilon(w)$ is the sign of $w \in W$.  Taking $t > 0$ and using (\ref{eqn:hc}), we can write
\begin{equation} \label{eqn:first-H-diff}
L_\mu(tx) - L_\lambda(tx) = \frac{\Delta_\gog(\rho)}{\Delta_\gog(tx)} \sum_{w \in W} \epsilon(w) \left( \frac{e^{t \langle w(\mu), x \rangle}}{\Delta_\gog(\mu)} - \frac{e^{t \langle w(\lambda), x \rangle}}{\Delta_\gog(\lambda)} \right).
\end{equation}
This expression is manifestly $W$-invariant in $\lambda$, $\mu$ and $x$, so we may assume without loss of generality that all three are dominant.  Then as $t \to +\infty$ we have:
\begin{equation} \label{eqn:H-asymp}
L_\mu(tx) - L_\lambda(tx) = \frac{\Delta_\gog(\rho)}{\Delta_\gog(tx)} \left( \frac{e^{t \langle \mu, x \rangle}}{\Delta_\gog(\mu)} - \frac{e^{t \langle \lambda, x \rangle}}{\Delta_\gog(\lambda)} \right) + (\textrm{lower-order terms}).
\end{equation}
Now suppose $\lambda \not \succeq \mu$.  Then $\mu$ lies outside the convex hull of the $W$-orbit of $\lambda$, so by the hyperplane separation theorem there is some $x_0 \in \tot$ and $C > 0$ such that $\langle \mu, x_0 \rangle > C$ while $\langle w(\lambda), x_0 \rangle < C$ for all $w \in W$.  By making a small perturbation to $x_0$ if necessary, we can ensure that $\Delta_\gog(x_0) \not = 0$.  Take $x$ in (\ref{eqn:H-asymp}) to be the dominant representative of the Weyl orbit of $x_0$.  Then we still have $\langle \mu, x \rangle > C$, $\langle \lambda, x \rangle < C$, and from (\ref{eqn:H-asymp}) we find:
$$ \lim_{t \to \infty} \Big[ L_\mu(tx) - L_\lambda(tx)  \Big] \ge \lim_{t \to \infty} \frac{\Delta_\gog(\rho)}{\Delta_\gog(tx)} \left( \frac{e^{t \langle \mu, x \rangle}}{\Delta_\gog(\mu)} - \frac{e^{t C}}{\Delta_\gog(\lambda)} \right) = \infty, $$
which implies $L_\mu(tx) > L_\lambda(tx)$ for some $t > 0$, so that (ii) cannot hold.

Now we remove the assumption that $\Delta_\gog(\lambda), \Delta_\gog(\mu) \not = 0$.  In this case the expression (\ref{eqn:first-H-diff}) may be singular, so we instead take a limit:
\begin{equation} \label{eqn:singular-H-diff}
L_\mu(tx) - L_\lambda(tx) = \lim_{\eta \to 0} \frac{\Delta_\gog(\rho)}{\Delta_\gog(tx)} \sum_{w \in W} \epsilon(w) \left( \frac{e^{t \langle w(\mu + \eta \rho), x \rangle}}{\Delta_\gog(\mu + \eta \rho)} - \frac{e^{t \langle w(\lambda + \eta \rho), x \rangle}}{\Delta_\gog(\lambda + \eta \rho)} \right).
\end{equation}
To evaluate this limit, we apply l'H\^opital's rule as many times as needed, treating the $\lambda$ and $\mu$ terms separately.  After $j$ applications to the $\lambda$ terms and $k$ applications to the $\mu$ terms for some $j, k \ge 0$, in place of (\ref{eqn:H-asymp}) we find:
\begin{equation} \label{eqn:singular-H-asymp}
L_\mu(tx) - L_\lambda(tx) = \frac{\Delta_\gog(\rho)}{\Delta_\gog(tx)} \left( \frac{t^k \langle \rho, x \rangle^k e^{t \langle \mu, x \rangle}}{\partial_\rho^k \Delta_\gog(\mu)} - \frac{t^j \langle \rho, x \rangle^j e^{t \langle \lambda, x \rangle}}{\partial_\rho^j \Delta_\gog(\lambda)} \right) +\ (\textrm{lower-order terms}),
\end{equation}
where $\partial_\rho^k \Delta_\gog(\mu) = \frac{d^k}{d \eta^k} \Delta_\gog(\mu + \eta \rho) \big |_{\eta = 0}$. The remainder of the argument then goes through as before, and we conclude that (ii) implies (i).

The other direction of the proof, (i) implies (ii), amounts to showing that for all $x \in \tot$, the function \mbox{$\lambda \mapsto L_\lambda(x)$}
is $W$-convex.  This function is clearly $W$-invariant, and by \mbox{\cite[Theorem 1]{GW},} a $W$-invariant, convex function is $W$-convex.  
It therefore remains only to show that $L_\lambda(x)$ is convex in $\lambda$, and for this it is sufficient to show midpoint convexity.  For $u, v \in \tot$ we have
\begin{multline*}
L_{\frac{1}{2}(u+v)}(x) = \int_{G} e^{ \langle \mathrm{Ad}_g (u + v)/2, x \rangle} dg =  \int_{G} \sqrt{e^{\langle \mathrm{Ad}_g u, x \rangle} e^{\langle \mathrm{Ad}_g v, x \rangle} } \,dg \\
\le \int_{G} \frac{1}{2}\big( e^{\langle \mathrm{Ad}_g u, x \rangle} + e^{\langle \mathrm{Ad}_g v, x \rangle} \big) \, dg = \frac{1}{2}L_u(x) + \frac{1}{2}L_v(x),
\end{multline*}
where in the final line we have applied the inequality of arithmetic and geometric means. This proves the theorem.
\end{proof}

\begin{remark} \normalfont
It is easily verified from the definition (\ref{eqn:L-def}) that $L_\lambda(x) = L_x(\lambda)$, so condition (ii) in Theorem \ref{thm:H-W-convexity} could equivalently be written: $$L_x(\lambda) \ge L_x(\mu) \textrm{ for all } x \in \tot.$$
\end{remark}

\section{Symmetric spaces}
\label{sec:symspaces}

This section contains our main results, which are majorization inequalities for spherical functions on Riemannian symmetric spaces.  After introducing some background on symmetric spaces and spherical functions, we state and prove separate inequalities for each of the three types of irreducible symmetric space.  In each case the theorem takes the form of an inequality between the pointwise values of any two spherical functions, which reflects the $W$-majorization order on the space of vectors that index the spherical functions.  As we explain below, these results imply Theorem \ref{thm:H-W-convexity} and discretizations
thereof, such as the Schur function inequality studied in \cite{CGS,Sra}.

\subsection{Background on symmetric spaces and spherical functions}
\label{sec:symspace-background}

Here we introduce only the minimum background required to state and prove the theorems.  We refer the reader to \cite[Appendices B and C]{OlPe} for a concise introduction to these topics, and to \cite{DS, GGA} for detailed references.  The definitions below mostly follow \cite[ch.~4]{GGA}

\begin{definition} \label{def:symspace} \normalfont
Let $G$ be a connected Lie group, and $K \subset G$ a compact subgroup.  We say that $(G, K)$ is a {\it symmetric pair} if $K$ is the fixed-point set of an involutive automorphism $\sigma: G \to G$.  For our purposes, a {\it Riemannian symmetric space} is a quotient $X = G/K$, where $(G, K)$ is a symmetric pair.  When $G$ is non-compact and semisimple with finite center, and $K$ is a maximal compact subgroup, we say that $X$ is of {\it non-compact type}.
\end{definition}

Below, the term ``symmetric space'' always means a Riemannian symmetric space as defined above.

\begin{definition} \label{def:spherical-funcs} \normalfont
Let $X = G/K$ be a Riemannian symmetric space, and write $[g] \in X$ for the image of $g \in G$ under the quotient map $G \to G/K$.  Let $\mathcal{D}(X)$ be the algebra of differential operators on $X$ that are invariant under all translations $[x] \mapsto [gx]$, $g \in G$.  A complex-valued function $\phi \in C^\infty(X)$ is called a {\it spherical function} if all of the following hold:
\begin{enumerate}
\item $\phi([\mathrm{id}]) = 1$,
\item $\phi([kx]) = \phi([x])$ for all $k \in K$,
\item $D \phi = \gamma_D \phi$ for each $D \in \mathcal{D}(X)$, where $\gamma_D$ is some complex eigenvalue.
\end{enumerate}
\end{definition}

Spherical functions play a central role in the theory of harmonic analysis on symmetric spaces, and many important families of special functions can be realized as spherical functions on some symmetric space.

\begin{example} \label{ex:cpt-gp-symspace} \normalfont
Let $G$ be a compact connected Lie group, and $K \subset G \times G$ the diagonal subgroup.  Then $(G \times G, K)$ is a symmetric pair, and we can identify $(G \times G)/K \cong G$ via $(g_1, g_2) K \mapsto g_1 g_2^{-1}$.  The spherical functions on $G$ are precisely the functions of the form $$\phi_\lambda(g) = \frac{\chi_\lambda(g)}{\dim V_\lambda},$$ where $V_\lambda$ is the irreducible representation of $G$ with highest weight $\lambda$, and $\chi_\lambda$ is its character.
\end{example}

\begin{example} \label{ex:alg-spherical-funcs} \normalfont
Let $G$ again be a compact connected Lie group.  If we regard its Lie algebra $\gog$ as an abelian Lie group, we can form the semidirect product $G \ltimes \gog$ with multiplication $(g_1, x_1) \cdot (g_2, x_2) = (g_1 g_2, \, \mathrm{Ad}_{g_1} x_2 + x_1)$.  Then $(G \ltimes \gog, G)$ is a symmetric pair, and we can identify $(G \ltimes \gog)/G \cong \gog$ via $(g, x) \mapsto \mathrm{Ad}_g x$.  Thus $\gog$ is a symmetric space, and the spherical functions on $\gog$ reduce to the Laplace transforms studied in Section \ref{sec:majorization},
$$L_{\lambda}(x) = \int_G e^{\langle \mathrm{Ad}_g \lambda, x \rangle} \, dg, \qquad x \in \gog, \quad \lambda \in \tot_\C,$$
where $\tot_\C$ is the complexification of a Cartan subalgebra $\tot \subset \gog$.
\end{example}

If $X$ is a symmetric space then its universal cover $\tilde X$ is also symmetric, and the spherical functions on $X$ may be identified with spherical functions on $\tilde X$ that are constant on the fibers of the covering map.  In this sense the spherical functions on $\tilde X$ subsume those on $X$, so that in what follows we may assume without loss of generality that $X$ is simply connected.  We say that $X$ is {\it irreducible} if it cannot be written as a nontrivial product of symmetric spaces.  A simply connected, irreducible symmetric space is always:
\renewcommand{\labelenumi}{(\arabic{enumi})}
\begin{enumerate}
\item of non-compact type; or,
\item a Euclidean space; or,
\item compact.
\end{enumerate}
\renewcommand{\labelenumi}{(\roman{enumi})}
These three types correspond respectively to the cases in which $X$ is negatively curved, flat, or positively curved.  There is a a well-known correspondence between the three types, which we now describe.  If $X^- = G/K$ is a symmetric space of non-compact type, $\sigma : G \to G$ is the associated involution fixing $K$, and $\gog = \mathrm{Lie}(G)$, then we have the Cartan decomposition $\gog = \mathfrak{k} + \mathfrak{p}$, where $\mathfrak{k} = \mathrm{Lie}(K)$ is the fixed-point set of $d \sigma$.  From these data we can construct both a Euclidean symmetric space and a compact symmetric space.  First define $\gog^+ = \mathfrak{k} + i \mathfrak{p} \subset \gog_\C$, which is the Lie algebra of the compact real form $G^+$ of $G$.  The symmetric space $X^+ = G^+/K$ is obviously compact.  Next define the algebra $\gog^0$, which is equal to $\gog$ as a vector space but is endowed with a different Lie bracket $[ \cdot, \cdot ]_0$ defined by
$$[x, y]_0 = \begin{cases}0, & x, y \in \mathfrak{p}, \\ [x,y], & \textrm{otherwise.} \end{cases}$$
Then the group $G^0 = \exp(\gog^0) \cong K \ltimes \mathfrak{p}$ acts on $\mathfrak{p}$ by affine transformations, $(k, p) \cdot x = \mathrm{Ad}_k x + p$, and $X^0 = G^0 / K \cong \mathfrak{p}$ is a Euclidean symmetric space.

Thus we have constructed a triple of symmetric spaces $(X^-, X^0, X^+)$ that belong respectively to the three types listed above.  Moreover, every simply connected, irreducible symmetric space occurs in such a triple.  In the following subsections, we study spherical functions on the spaces $X^-$, $X^0$, and $X^+$.

\subsection{Symmetric spaces of non-compact type}

When $X^- = G/K$ is a symmetric space of non-compact type, the spherical functions admit a convenient integral representation, due to Harish-Chandra \cite{HC2, HC3}.  The version that we use here is proved in \cite[ch.~4, Theorem 4.3]{GGA}.  Let $G = NAK$, $\gog = \mathfrak{n} + \mathfrak{a} + \mathfrak{k}$ be the Iwasawa decompositions of $G$ and $\gog$.  For $g \in G$, let $a(g)$ be the unique element of $\mathfrak{a}$ such that $g \in N e^{a(g)} K$. The Killing form $\langle \cdot, \cdot \rangle$ on $\gog$ restricts to an inner product on $\mathfrak{a}$.  For $\alpha \in \mathfrak{a}$, define $$\gog_\alpha = \{ \ x \in \gog \ | \ [h, x] = \langle \alpha, h \rangle x \ \textrm{for all} \ h \in \mathfrak{a} \ \}.$$
The {\it restricted root system} $\Phi$ of $X^-$ consists of all nonzero $\alpha \in \mathfrak{a}$ for which $\gog_\alpha$ is nontrivial.  Fix a choice $\Phi^+$ of positive roots, and let $W$ be the Weyl group generated by reflections in the root hyperplanes.  For $\lambda, \mu \in \mathfrak{a}$, we again write $\lambda \succeq \mu$ to indicate that $\lambda$ $W$-majorizes $\mu$.  For $\alpha \in \Phi$, define $m_\alpha = \dim \gog_\alpha$, and set $\rho = \frac{1}{2} \sum_{\alpha \in \Phi^+} m_\alpha \alpha$.  Write $dk$ for the normalized Haar measure on $K$.

\begin{theorem}[Harish-Chandra] \label{thm:hc-integral-spherical}
The spherical functions on $X^-$ are exactly the functions of the form
\begin{equation} \label{eqn:hc-integral-spherical}
\phi^-_\lambda([g]) = \int_K e^{\langle i \lambda + \rho, a(kg) \rangle} dk, \qquad g \in G,
\end{equation}
as $\lambda$ ranges over $\mathfrak{a}_\C$. Moreover, two such functions $\phi^-_\lambda$ and $\phi^-_\mu$ are identical if and only if $\mu = w(\lambda)$ for some $w \in W$.
\end{theorem}

The following theorem is the main result of this paper.

\begin{theorem} \label{thm:spherical-majorization}
Let $X^- = G/K$ be a Riemannian symmetric space of non-compact type. For any $\lambda, \mu \in \mathfrak{a}$, the following are equivalent:
\begin{enumerate}
\item $\lambda \succeq \mu$,
\item $\phi^-_{i \lambda}(x) \ge \phi^-_{i \mu}(x)$ for all $x \in X$.
\end{enumerate}
\end{theorem}

\begin{proof}
The argument generalizes the proof of Theorem \ref{thm:H-W-convexity}.  We first show that (ii) implies (i) by proving the contrapositive, and then that (i) implies (ii) using the integral representation (\ref{eqn:hc-integral-spherical}) for the spherical functions.

Suppose $\lambda \not \succeq \mu$.  Since the map $\lambda \mapsto -\lambda$ is an isometry of $\mathfrak{a}$, we have $\lambda \succeq \mu$ if and only if $- \lambda \succeq - \mu$.  Accordingly, to prove that (ii) implies (i), it suffices to show that $\phi^-_{- i \mu}(x) > \phi^-_{- i \lambda}(x)$ for some $x \in X$.  By hyperplane separation, we can obtain $y \in \mathfrak{a}$ and $C_1 > 0$ such that $\langle \mu, y \rangle > C_1$ and $\langle w(\lambda), y \rangle < C_1$ for all $w \in W$.  Clearly both of these inequalities still hold if we replace $y$ with the dominant representative of its Weyl orbit, and by Theorem \ref{thm:hc-integral-spherical} we have $\phi^-_{i \lambda} = \phi^-_{i w(\lambda)}$ for $w \in W$.  Therefore without loss of generality we may take all three of $\lambda$, $\mu$ and $y$ to be dominant.

With these assumptions, we will study the asymptotic behavior of the spherical functions $\phi^-_{- i \lambda}$ and $\phi^-_{- i \mu}$ at infinity.  This topic is well understood; see e.g.~\cite{Duis-asymp}.  In particular we have the following sharp estimate as $t \to +\infty$, which also follows directly from (\ref{eqn:symspace-HGF}) and (\ref{eqn:HGF-asymp}) below:
\begin{equation} \label{eqn:spherical-asymp}
\phi^-_{- i \lambda}([e^{ty}]) \ \asymp \ e^{t \langle \lambda - \rho, \, y \rangle} \prod_{\alpha \in \Phi^+_\lambda} (1 + 2 t \langle \alpha, y \rangle ),
\end{equation}
where
\begin{equation} \label{eqn:phi-plus-lambda}
\Phi^+_\lambda = \Big \{ \ \alpha \in \Phi^+ \ \Big | \ \frac{1}{2} \alpha \not \in \Phi^+, \ \langle \alpha, \lambda \rangle = 0  \ \Big \}.
\end{equation}
Here the notation $f(t) \asymp g(t)$ means that there exist constants $c, C, T > 0$ such that $c g(t) < f(t) < Cg(t)$ for all $t > T$.  Thus the estimate (\ref{eqn:spherical-asymp}) implies that
$$\phi^-_{- i \mu}([e^{ty}]) - \phi^-_{- i \lambda}([e^{ty}]) \ > \ e^{-t \langle \rho, y \rangle} \bigg(C_2 \, e^{t \langle \mu, y \rangle} - C_3 \, e^{tC_1} \prod_{\alpha \in \Phi^+_\lambda} (1 + 2 t \langle \alpha, y \rangle ) \bigg)$$
for large $t$ and some constants $C_2, C_3 > 0$.  For $t$ sufficiently large, the quantity on the right-hand side above is positive, proving that $\phi^-_{- i \mu}(x) > \phi^-_{- i \lambda}(x)$ for some $x \in X$, as desired.

We next prove that (i) implies (ii). It suffices to show that the function $f_x(\lambda) = \phi^-_{i \lambda}(x)$ is $W$-convex.  As in the proof of Theorem \ref{thm:H-W-convexity}, we use the result of \cite[Theorem 1]{GW}, which states that a $W$-invariant, convex function is $W$-convex.  By Theorem \ref{thm:hc-integral-spherical}, $f_x$ is $W$-invariant, so we need only prove that $f_x$ is convex, for which it suffices to check midpoint convexity.  Write $x = [g]$ for some $g \in G$.  Using the integral representation (\ref{eqn:hc-integral-spherical}) and the inequality of arithmetic and geometric means, we find:
\begin{align*}
f_{x} \bigg(\frac{1}{2}(\lambda + \mu) \bigg) &= \int_K e^{\langle \rho - (\lambda + \mu)/2, \, a(kg) \rangle} dk \\
&=  \int_K e^{\langle \rho, a(kg) \rangle} \sqrt{ e^{- \langle \lambda, a(kg) \rangle} e^{- \langle \mu, a(kg) \rangle} } \, dk \\
& \ge \frac{1}{2} \int_K e^{\langle \rho, a(kg) \rangle} ( e^{- \langle \lambda, a(kg) \rangle} + e^{- \langle \mu, a(kg) \rangle} ) \, dk \\
&= \frac{1}{2} \big( f_x(\lambda) + f_x(\mu) \big),
\end{align*}
which shows that $f_x$ is convex, completing the proof.
\end{proof}

\begin{remark} \label{rem:bounded-phi} \normalfont
There are parallels between Theorem \ref{thm:spherical-majorization} and a celebrated result of Helgason and Johnson \cite{HJ}, which classifies the bounded spherical functions on $X^-$.  Write $\mathrm{Re} \lambda,\, \mathrm{Im} \lambda \in \mathfrak{a}$ for the real and imaginary parts of $\lambda \in \mathfrak{a}_\C$.  Although Helgason and Johnson do not use the terminology of $W$-majorization, their theorem states that $\phi^-_\lambda$ is bounded if and only if $\rho \succeq \mathrm{Im} \lambda.$  More recently, this result has been generalized to a classification of bounded Heckman--Opdam hypergeometric functions \cite{NPP}.

Theorem \ref{thm:spherical-majorization} and the Helgason--Johnson theorem are independent but related.  The boundedness of $\phi^-_\lambda$ when $\rho \succeq \mathrm{Im} \lambda$ follows from Theorem \ref{thm:spherical-majorization} once one observes from the integral formula (\ref{eqn:hc-integral-spherical}) that $|\phi^-_\lambda| \le \phi^-_{i \, \mathrm{Im} \lambda}$ and that $\phi^-_{i \rho} \equiv 1.$  On the other hand, Theorem \ref{thm:spherical-majorization} by itself does not imply that all other $\phi^-_\lambda$ are unbounded (though this can be deduced from the estimate (\ref{eqn:spherical-asymp})), while the Helgason--Johnson theorem does not tell us that $W$-majorization actually induces a partial order on {\it all} $\phi^-_\lambda$ with $\lambda \in i \mathfrak{a}$.
\end{remark}

\subsection{Euclidean symmetric spaces}

The spherical functions on the Euclidean symmetric space $X^0 \cong \mathfrak{p}$ are precisely the functions
\begin{equation} \label{eqn:spherical-typeI}
\phi^0_\lambda(x) = \lim_{\varepsilon \to 0} \phi^-_{\lambda / \varepsilon}([e^{\varepsilon x}]) = \int_K e^{i \langle \lambda, \mathrm{Ad}_k x \rangle} dk, \qquad x \in \mathfrak{p},
\end{equation}
as $\lambda$ ranges over $\mathfrak{a}_\C$; see \cite[ch.~4, Proposition 4.8]{GGA}.  Taking the limit (\ref{eqn:spherical-typeI}) in the proof of Theorem \ref{thm:spherical-majorization}, we obtain the following.

\begin{proposition} \label{prop:euclidean-majorization}
Let $X^0$ be a Euclidean symmetric space.  For any $\lambda, \mu \in \mathfrak{a}$, the following are equivalent:
\begin{enumerate}
\item $\lambda \succeq \mu$,
\item $\phi^0_{i \lambda}(x) \ge \phi^0_{i \mu}(x)$ for all $x \in X$.
\end{enumerate}
\end{proposition}

In particular, Theorem \ref{thm:H-W-convexity} is a special case of Proposition \ref{prop:euclidean-majorization}, corresponding to the Euclidean symmetric space described in Example \ref{ex:alg-spherical-funcs}.

\subsection{Compact symmetric spaces} \label{sec:cpt-sym}

We now consider the compact symmetric space $X^+ = G^+/K$.  Let $V_\lambda$ be the irreducible $G^+$-representation with highest weight $\lambda$, and $\chi_\lambda$ its character.  If $V_\lambda$ contains a nontrivial $K$-fixed vector, we say that $V_{\lambda}$ is a {\it spherical} representation and $\lambda$ is a spherical highest weight.  By \cite[ch.~4, Theorem 4.2]{GGA}, the spherical functions on $X^+$ are precisely the functions
\begin{equation} \label{eqn:cpt-spherical}
\phi^+_\lambda([g]) = \int_K \chi_\lambda(g^{-1}k) \, dk, \qquad g \in G^+,
\end{equation}
where $\chi_\lambda$ is the character of an irreducible spherical representation of $G^+$.

Here we depart in two ways from the conventions used above in Section \ref{sec:majorization}.  First, we now use the notation $\langle \cdot, \cdot \rangle$ to indicate the Killing form, which restricts to a {\it negative}-definite form on $\gog^+$ rather than an inner product.  Second, we now regard the roots and weights of $G^+$ as {\it imaginary}-valued linear functionals on a Cartan subalgebra $\tot \subset \gog^+$ with $i \mathfrak{a} \subset \tot$.  We then use the Killing form to identify the weights and roots with elements of $i \tot$.

With these conventions, the spherical highest weights of $G^+$ correspond to certain lattice points in $\mathfrak{a} \subset i \tot$; see \cite[ch.~5 \textsection 4]{GGA}.  By \cite[ch.~5, Theorem 4.1 and Corollary 4.2]{GGA}, if $G^+$ is simply connected and semisimple then the spherical highest weights are exactly those $\lambda \in \mathfrak{a}$ satisfying
\begin{equation} \label{eqn:lambda-reqs}
\frac{\langle \lambda, \alpha \rangle}{\langle \alpha, \alpha \rangle} \in \Z_{\ge 0} \quad \textrm{for all } \alpha \in \Phi^+,
\end{equation}
where $\Phi^+$ are the positive restricted roots of $X^-$.  Here as well, given $\lambda, \mu \in \mathfrak{a}$, we write $\lambda \succeq \mu$ to indicate that $\lambda$ $W$-majorizes $\mu$, where $W$ is the Weyl group of the {\it restricted} root system of $X^-$ (rather than the root system of $G^+$).

The function $\phi^+_\lambda$ can be analytically continued to the complexification $G_\C$, so that we may evaluate $\phi^+_\lambda([e^x])$ for any $x \in \tot_\C$.  We then have the following majorization inequality.

\begin{proposition} \label{prop:compact-majorization}  Let $\lambda, \mu \in \mathfrak{a}$ be two spherical highest weights of $G^+$. The following are equivalent:
\begin{enumerate}
\item $\lambda \succeq \mu$,
\item $\phi^+_\lambda([e^{ix}]) \ge \phi^+_\mu([e^{ix}])$ for all $x \in \tot$.
\end{enumerate}
\end{proposition}

\begin{proof}
Consider the spherical function $\phi^-_{-i(\lambda - \rho)}$ on $X^-$, regarded as a function on the non-compact group $G$.  When $\lambda$ is a spherical highest weight, this function also admits an analytic continuation to $G_\C$, which coincides with $\phi^+_\lambda$; see \cite[\textsection 4]{LV}.  Since $\tot = \tot \cap \mathfrak{k} + i \mathfrak{a}$ and $[e^{y + x}] = [e^x] \in X^+$ for $y \in \tot \cap \mathfrak{k}$, we can take $x \in i \mathfrak{a}$, so that $e^{ix} \in G$.  The desired result is then immediate from Theorem \ref{thm:spherical-majorization}.
\end{proof}

\subsection{Application to orthogonal polynomials}
Let us explain how Proposition \ref{prop:compact-majorization} implies the results of \cite{CGS,Sra} relating majorization and Schur polynomials, which we stated above in Theorem \ref{thm:CGSS}.  Young diagrams with $N$ rows index the irreducible polynomial representations of $\U(N)$.  If $\lambda$ is a Young diagram and $V_\lambda$ is the corresponding irreducible representation, we can identify $\R^N$ with a Cartan subalgebra in $\mathfrak{u}(N)$ such that
\begin{equation} \label{eqn:schur-char}
S_\lambda(e^{iy_1}, \hdots, e^{iy_N}) = \frac{\chi_\lambda(e^y)}{\dim V_\lambda}, \qquad y = (y_1, \hdots, y_N) \in \R^N.
\end{equation}
If we then regard the group $\U(N)$ as a compact symmetric space as in Example \ref{ex:cpt-gp-symspace}, we find
\begin{equation} \label{eqn:spherical-schur}
\phi^+_\lambda([e^{iy}]) = S_\lambda(e^{y_1}, \hdots, e^{y_N}),
\end{equation}
that is, in an appropriate choice of coordinates, the spherical functions are normalized Schur polynomials.  Writing $x_i = e^{y_i}$, Proposition \ref{prop:compact-majorization} then yields Theorem \ref{thm:CGSS} under the stricter assumption that all  
$x_1, \hdots, x_N > 0$.  Since Schur polynomials are continuous, we can relax this to $x_1, \hdots, x_N \ge 0$, which proves Theorem \ref{thm:CGSS} in full.  A related characterization of \emph{weak majorization} in terms of Schur polynomials was recently obtained by Khare and Tao in \cite{KhTao}.

Schur polynomials are not unique in this regard.  Many families of orthogonal polynomials, such as the real and quaternionic zonal polynomials, can be realized as spherical functions on compact symmetric spaces.  In fact, Vretare \cite{LV} showed that {\it every} compact symmetric space gives rise to an associated family of orthogonal polynomials, which express the spherical functions via a relation similar to (\ref{eqn:spherical-schur}).  In all such cases, Proposition \ref{prop:compact-majorization} gives an inequality for the orthogonal polynomials that is analogous to Theorem \ref{thm:CGSS}.

\begin{remark} \normalfont
One can also deduce Theorem \ref{thm:CGSS} directly from Theorem \ref{thm:H-W-convexity} by using the Kirillov character formula for compact groups \cite[ch.~5]{AK} to write the Schur polynomials in terms of integrals of the form (\ref{eqn:L-def}). Since the arguments in (\ref{eqn:L-def}) are not constrained to lie in the weight lattice, in the case $G=\U(N)$ Theorem \ref{thm:H-W-convexity} in fact entails an extension of Theorem \ref{thm:CGSS} to generalized Schur polynomials with non-integer exponents, which was previously observed by Khare and Tao \cite{KhTao2}.  Corresponding statements for the characters of other compact groups are also immediate from Theorem \ref{thm:H-W-convexity}.
\end{remark}

\section{Hypergeometric functions}
\label{sec:conjecture}

The Heckman--Opdam hypergeometric functions are a family of special functions associated to root systems, which generalize the classical Gauss hypergeometric function to higher dimensions.  They are eigenfunctions of the hyperbolic quantum Calogero--Sutherland Hamiltonian and were introduced in the paper \cite{RSHF1} in order to prove the complete integrability of quantum Calogero--Sutherland models.  Many special functions of interest can be expressed via limits or specializations of Heckman--Opdam hypergeometric functions, including the spherical functions on all Riemannian symmetric spaces of non-compact type.  Also in \cite{RSHF1}, Heckman and Opdam defined the multivariable Jacobi polynomials, now known as Heckman--Opdam polynomials.  These are closely related to hypergeometric functions and generalize numerous widely studied families of orthogonal polynomials, such as Schur and Jack polynomials.

In this final section, we conjecture that Heckman--Opdam hypergeometric functions satisfy a fundamental monotonicity property with respect to $W$-majorization.  If true, this conjecture would unify and generalize all of the majorization results discussed in this paper.  We prove one of the two directions of implication that comprise the conjecture, and we show that the full conjecture holds in rank one.

Just as Heckman--Opdam hypergeometric functions generalize the spherical functions $\phi^-_{\lambda}$ on different symmetric spaces of non-compact type, the Heckman--Opdam polynomials generalize of the functions $\phi^+_\lambda$, up to some differences in normalization.  Similarly, another related class of functions, the generalized Bessel functions, interpolate between the functions $\phi^0_\lambda$ on different Euclidean symmetric spaces.  If the conjecture is true, then analogous results hold for both generalized Bessel functions and Heckman--Opdam polynomials.  In particular, we show below that the conjecture would immediately imply an analogue for Heckman--Opdam polynomials of the Schur polynomial inequality proved in \cite{CGS,Sra}.

To define the Heckman--Opdam hypergeometric functions, we first must fix some preliminary data.  Here we largely follow the conventions of Anker \cite{AnkerDunklNotes} and Heckman and Schlichtkrul \cite{HS}.  Let $V \cong \R^r$ be a Euclidean space with inner product $\langle \cdot, \cdot \rangle$, $\Phi \subset V$ a (crystallographic) root system spanning $V$, and $W$ the Weyl group acting on $V$ by reflections in the root hyperplanes.  The Heckman--Opdam hypergeometric function $F_{k,\lambda}$ depends on a point $\lambda$ in the complexification $V_\C$, as well as on a {\it multiplicity parameter}, which is a function $k: \Phi \to \C$ such that $k_{w \cdot \alpha} = k_\alpha$ for $w \in W$.  Unless stated otherwise, we assume in what follows that $\lambda \in V$ and that all $k_\alpha$ are nonnegative real numbers.

We now define $F_{k,\lambda}$ in terms of solutions to certain differential-difference equations.  Fix a choice of positive roots $\Phi^+$.  For $\alpha \in \Phi^+$ let $s_\alpha$ be the reflection through the hyperplane $\{x \in V \ | \ \langle \alpha, x \rangle = 0 \},$ and define $\rho^{(k)} = \frac{1}{2} \sum_{\alpha \in \Phi^+} k_\alpha \alpha.$

\begin{definition} \label{def:cherednik-op} \normalfont
For $y \in V$, the \emph{Cherednik operator} $D_{k,y}$ is the differential-difference operator
\begin{equation}
D_{k,y} = \partial_y + \sum_{\alpha \in \Phi^+} \langle y, \alpha \rangle k_\alpha \frac{1}{1 - e^{-\alpha}} (1 - s_\alpha) - \langle y, \rho^{(k)} \rangle.
\end{equation}
\end{definition}

The Cherednik operators were originally defined and studied in \cite{Ch1, Ch2}. For details of their properties, we refer the reader to these papers as well as to \cite[\textsection 4]{AnkerDunklNotes} and \cite[\textsection 2]{Op}.  Here we need only the following fact: when all $k_\alpha$ are nonnegative, for any $\lambda \in V_\C$ there is a unique smooth function $G_{k,\lambda}$ on $V$ satisfying the system of differential-difference equations
\begin{equation} \label{eqn:opdam-HG-eqn}
D_{k,y} G_{k,\lambda} = \langle y, \lambda \rangle G_{k,\lambda} \quad \textrm{for all } y \in V
\end{equation}
and normalized so that $G_{k,\lambda}(0) = 1$.

\begin{definition} \label{def:hgf} \normalfont
For any nonnegative real multiplicity parameter $k$ and any $\lambda \in V_\C$, the {\it Heckman--Opdam hypergeometric function} $F_{k, \lambda}$ is defined as
\begin{equation} \label{eqn:hgf-def}
F_{k, \lambda}(x) = \frac{1}{|W|} \sum_{w \in W} G_{k,\lambda}(w(x)).
\end{equation}
\end{definition}

The functions $F_{k, \lambda}$ unify and interpolate between many widely studied special functions, as illustrated in the following examples.

\begin{example} \label{ex:rank1} \normalfont
In the 1-dimensional case where $V \cong \R$, the root system $\Phi$ can be either $A_1$ or $BC_1$.  For $BC_1$ there are two Weyl orbits $\{ \pm 1 \}$, $\{ \pm 2 \} \subset \R$, and the Heckman--Opdam hypergeometric function reduces to the Gauss hypergeometric function:
\begin{equation} \label{eqn:HO-to-Gauss}
F_{k, \lambda}(x) = {}_{2}F_1 \Big( \frac{k_1}{2} + k_2 + \lambda, \, \frac{k_1}{2} + k_2 - \lambda ; \, k_1 + k_2 + \frac{1}{2} ; \, - \sinh^2 \frac{x}{2} \Big).
\end{equation}
The Heckman--Opdam hypergeometric function for $A_1$ corresponds to the special case \mbox{$k_2 = 0$}.
\end{example}

\begin{example} \label{ex:symspace-HGF} \normalfont
Suppose $\Psi$ is the restricted root system of a symmetric space $X = G/K$ of non-compact type, and $m_\alpha = \dim \mathfrak{g}_\alpha$ for each $\alpha \in \Psi$.  Take $V = \mathfrak{a}$, $\Phi = \{ 2 \alpha \ | \ \alpha \in \Psi \}$ and $k_{2\alpha} = \frac{1}{2} m_{\alpha}$.  Then
\begin{equation} \label{eqn:symspace-HGF}
\phi^-_{\lambda}([e^x]) = F_{k, i \lambda }(x), \qquad x \in \mathfrak{a}, \quad \lambda \in \mathfrak{a}_\C.
\end{equation}
See \cite[\textsection 4]{AnkerDunklNotes}.
\end{example}

\begin{example} \label{ex:gen-bessel} \normalfont
The generalized Bessel function $J_{k, \lambda}$ on $V$ can be obtained as the {\it rational limit} of $F_{k,\lambda}$:
\begin{equation}
J_{k, \lambda}(x) = \lim_{\varepsilon \to 0} F_{k, \lambda / \varepsilon}(\varepsilon x).
\end{equation}
From the previous example and the relation (\ref{eqn:spherical-typeI}), it is clear that $J_{k, \lambda}$ generalizes the spherical functions on Euclidean symmetric spaces in the same way that $F_{k, \lambda}$ generalizes the spherical functions on symmetric spaces of non-compact type.  See \cite[\textsection 3 and \textsection 4.4]{AnkerDunklNotes} for details on generalized Bessel functions and the rational limit.
\end{example}

For any multiplicity parameter $k$, we define the function
\begin{equation} \label{eqn:delta-k-def}
\delta_k(x) = \prod_{\alpha \in \Phi^+} (e^{\langle \alpha, x \rangle/2} - e^{- \langle \alpha,x \rangle/2})^{k_\alpha}.
\end{equation}

\begin{example} \label{ex:HGF-to-L} \normalfont
When $\Phi$ is reduced, we can identify $V$ with a Cartan subalgebra $\tot$ of a compact semisimple Lie algebra $\gog$ with root system $\Phi$. We write $k = \vec 1$ for the multiplicity parameter with $k_\alpha = 1$ for all $\alpha \in \Phi$. Then
\begin{equation} \label{eqn:HGF-to-L}
F_{\vec 1, \lambda}(x) = \frac{\Delta_\gog(x)}{\delta_{\vec 1}(x)} L_{\lambda}(x), \qquad x \in \tot.
\end{equation}
\end{example}

The notion of $W$-majorization is defined in this setting just as in Definition \ref{def:W-majorization}.  We conjecture the following monotonicity property for Heckman--Opdam hypergeometric functions with respect to $W$-majorization.

\begin{conjecture} \label{conj:HGF-majorization}
Let $\lambda, \mu \in V$ and let $k$ be a nonnegative real multiplicity parameter.  The following are equivalent:
\begin{enumerate}
\item $\lambda \succeq \mu$,
\item $F_{k,\lambda}(x) \ge F_{k,\mu}(x)$ for all $x \in V$.
\end{enumerate}
\end{conjecture}

Here we show one half of the conjecture, namely that (ii) implies (i), using known sharp asymptotics for $F_{k,\lambda}$.  We then give an elementary proof of the conjecture in rank one, where it amounts to an inequality for the Gauss hypergeometric function.

\begin{proposition} \label{prop:HGF-majorization-onlyif}
Take $\lambda, \mu$ and $k$ as in Conjecture \ref{conj:HGF-majorization}.  If $F_{k,\lambda}(x) \ge F_{k,\mu}(x)$ for all $x \in V$, then $\lambda \succeq \mu$.
\end{proposition}

\begin{proof}
Again we show the contrapositive.  Suppose $\lambda \not \succeq \mu$, and again use hyperplane separation to obtain a $y \in V$ and $C_1 > 0$ such that $\langle \mu, y \rangle > C_1$ and $\langle w(\lambda), y \rangle < C_1$ for all $w \in W$.  Without loss of generality, we take $\lambda$, $\mu$ and $y$ to be dominant.  We have the following sharp asymptotic estimate, due to Schapira \cite[Remark 3.1]{SchapiraHGF} and Narayanan--Pasquale--Pusti \cite[Theorem 3.4]{NPP}:
\begin{equation} \label{eqn:HGF-asymp}
F_{k,\lambda}(ty) \ \asymp \ e^{t \langle \lambda - \rho^{(k)}, \, y \rangle} \prod_{\alpha \in \Phi^+_\lambda} (1 + t \langle \alpha, y \rangle )
\end{equation}
as $t \to +\infty,$ with $\Phi^+_\lambda$ as defined above in (\ref{eqn:phi-plus-lambda}) and $\asymp$ as defined immediately after. We thus find that
$$F_{k, \mu}(ty) - F_{k, \lambda}(ty) \ > \ e^{-t \langle \rho^{(k)}, \, y \rangle} \bigg( C_2 \, e^{t \langle \mu, y \rangle} - C_3 \, e^{t C_1} \prod_{\alpha \in \Phi^+_\lambda} (1 + t \langle \alpha, y \rangle ) \bigg)$$
for large $t$ and some constants $C_2, C_3 > 0$.  For $t$ sufficiently large, the quantity on the right-hand side above is positive, which implies that $F_{k,\lambda}(x) < F_{k,\mu}(x)$ for some $x \in V$, completing the proof.
\end{proof}

\begin{remark} \label{rem:HO-integral-formulae} \normalfont
As in the proof of Theorem \ref{thm:spherical-majorization}, to complete the proof of Conjecture \ref{conj:HGF-majorization} it suffices to check that the function $\lambda \mapsto F_{k,\lambda}(x)$ is midpoint-convex.  However, integral representations analogous to (\ref{eqn:hc-integral-spherical}) for general Heckman--Opdam hypergeometric functions are not known, so we cannot directly apply the same technique.  Although there are known integral expressions in certain cases where the multiplicity parameter does not correspond to a symmetric space (see e.g.~\cite{RoslerVoit, Sun}), these are more complicated than (\ref{eqn:hc-integral-spherical}) and have so far resisted a similar analysis. A more promising approach to a general proof of Conjecture \ref{conj:HGF-majorization} might be to use hypergeometric differential equations, as illustrated in the following proposition.  Another approach could be to analyze a series expansion for the hypergeometric function as in \cite{NPP}.
\end{remark}

\begin{proposition} \label{prop:HGF-rank1}
When $\dim V = 1$, Conjecture \ref{conj:HGF-majorization} is true.
\end{proposition}

\begin{proof}
In light of Proposition \ref{prop:HGF-majorization-onlyif}, we need only show that (i) implies (ii) in Conjecture \ref{conj:HGF-majorization}.  Following the discussion in Example \ref{ex:rank1}, it is sufficient to consider the case $\Phi = BC_1$.  It is a classical result that the Gauss hypergeometric function $F(z) = {}_2 F_1(a, b; c; z)$ satisfies Euler's hypergeometric equation:
$$ z(1-z) \frac{d^2F}{dz^2} + [c - (a+b+1)z] \frac{dF}{dz} - ab \, F = 0.$$
Comparing to (\ref{eqn:HO-to-Gauss}), we find:
\begin{equation} \label{eqn:HGDE-rank1}
F''_{k,\lambda}(x) + \Big( k_1 \coth \frac{x}{2} + 2 k_2 \coth x \Big) F'_{k, \lambda}(x) + \Big[ \Big(\frac{k_1}{2} + k_2 \Big)^2 - \lambda^2 \Big] F_{k,\lambda}(x) = 0,
\end{equation}
for $\lambda, x \in \R$.  The function $F_{k,\lambda}$ is determined by (\ref{eqn:HGDE-rank1}) and by the initial conditions
\begin{equation} \label{eqn:rank1-ICs}
F'_{k,\lambda}(0) = 0, \qquad F_{k, \lambda}(0) = 1.
\end{equation}
The above initial conditions follow directly from the definition (\ref{eqn:hgf-def}) in rank one: we must have $F'_{k,\lambda}(0) = 0$ because $F_{k, \lambda}$ is $W$-invariant (i.e.~even), and we must have $F_{k, \lambda}(0) = 1$ due to the stipulation that $G_{k, \lambda}(0) = 1$.

Suppose $\lambda \succeq \mu$, which in dimension one just means that $|\lambda| \ge |\mu|$.  Since the equation (\ref{eqn:HGDE-rank1}) depends only on $\lambda^2$ and not on the sign of $\lambda$, we find $F_{k, -\lambda} = F_{k, \lambda}$, so we can in fact take $|\lambda| > |\mu|$.  We will show that $F_{k, \lambda}(x) \ge F_{k, \mu}(x)$ for all $x \in \R$, with equality only at $x = 0$.

Since $F_{k, \lambda}$ is even, it suffices to consider $x \ge 0$.  From (\ref{eqn:HGDE-rank1}) and the initial conditions (\ref{eqn:rank1-ICs}), we have:
$$F_{k, \lambda}(0) = F_{k, \mu}(0) = 1,$$
$$F'_{k,\lambda}(0) = F'_{k,\mu}(0) = 0,$$
$$F''_{k,\lambda}(0) = \lambda^2 - \Big(\frac{k_1}{2} + k_2 \Big)^2 > \mu^2 - \Big(\frac{k_1}{2} + k_2 \Big)^2 = F''_{k,\mu}(0).$$
Therefore there is some $\varepsilon > 0$ such that for all $x \in (0, \varepsilon)$,
$$F_{k, \lambda}(x) > F_{k, \mu}(x), \qquad F'_{k, \lambda}(x) > F'_{k, \mu}(x), \qquad F''_{k, \lambda}(x) > F''_{k, \mu}(x).$$
Consider the set $E = \{ x > 0 \ | \ F_{k, \lambda}(x) = F_{k, \mu}(x) \}.$  If $E$ is empty, then $F_{k, \lambda}(x) > F_{k, \mu}(x)$ for all $x > 0$, and there is nothing to prove.  Assume for the sake of contradiction that $E$ is not empty, and let $x_0 = \inf E.$  Clearly $x_0 > \varepsilon.$  Since $F_{k, \lambda}$ and $F_{k, \mu}$ are continuous, we have $F_{k, \lambda}(x_0) = F_{k, \mu}(x_0)$, and
\begin{equation} \label{eqn:F-x0}
F_{k, \lambda}(x) > F_{k, \mu}(x) \textrm{ for all } x \in (0, x_0).
\end{equation}
However, since $$F_{k, \lambda}(x_0) = 1 + \int_{0}^{x_0} F'_{k, \lambda}(\tau) \, d\tau, \qquad F_{k, \mu}(x_0) = 1 + \int_{0}^{x_0} F'_{k, \mu}(\tau) \, d\tau,$$ and $F'_{k, \lambda} > F'_{k, \mu}$ on $(0, \varepsilon)$, in order to have $F_{k, \lambda}(x_0) = F_{k, \mu}(x_0)$ there must be some $x_1 \in (\varepsilon, x_0)$ such that $F'_{k, \lambda}(x_1) < F'_{k, \mu}(x_1)$.  By the intermediate value theorem and by (\ref{eqn:F-x0}), there must then be some $x_2 \in (\varepsilon, x_1)$ such that all of the following hold:
\[
\textrm{(a) } F'_{k, \lambda}(x_2) = F'_{k, \mu}(x_2), \qquad
\textrm{(b) } F''_{k, \lambda}(x_2) < F''_{k, \mu}(x_2), \qquad
\textrm{(c) } F_{k, \lambda}(x_2) > F_{k, \mu}(x_2).
\]
Subtracting (\ref{eqn:HGDE-rank1}) from the corresponding equation for $F_{k, \mu}$ and using condition (a) above to cancel terms, we find
$$ F''_{k,\mu}(x_2) - F''_{k,\lambda}(x_2) + \Big[ \Big(\frac{k_1}{2} + k_2 \Big)^2 - \mu^2 \Big] F_{k,\mu}(x_2) - \Big[ \Big(\frac{k_1}{2} + k_2 \Big)^2 - \lambda^2 \Big] F_{k,\lambda}(x_2) = 0.$$
But by (b), (c), and the assumption that $|\lambda| > |\mu|$, the left-hand side above must be positive, yielding a contradiction and completing the proof.
\end{proof}

We next turn our attention from hypergeometric functions to the closely related family of Heckman--Opdam polynomials, which we now define.  For $\alpha \in \Phi$, write $$\alpha^\vee = \frac{2 \alpha}{\langle \alpha, \alpha \rangle}.$$  The fundamental weights $\omega_1, \hdots, \omega_r$ are defined by $\langle \omega_i, \alpha_j^\vee \rangle = \delta_{ij},$ where $\alpha_1, \hdots, \alpha_r$ are the simple roots.  They span the {\it weight lattice} $P \subset V$.  The {\it dominant integral weights} are the lattice points $P^+ \subset P$ that lie in the dominant Weyl chamber.

The Heckman--Opdam polynomials $P_{k,\lambda}$ depend on a nonnegative multiplicity parameter $k$ and a dominant integral weight $\lambda \in P^+$.  They are elements of $\R[P]$, the group algebra of the weight lattice, and are therefore polynomials in an abstract-algebraic sense.  However, it is typical to identify $\R[P]$ with the algebra spanned by the functions $e^{\langle \lambda, x \rangle}$, $\lambda \in P$, so that as functions on $V$ the Heckman--Opdam polynomials are actually {\it exponential} polynomials.

We write an element $f \in \R[P]$ as $f = \sum_{\lambda \in P} f_\lambda e^\lambda$, where only finitely many $f_\lambda$ are nonzero, and set $$\bar f = \sum_{\lambda \in P} f_{-\lambda} e^\lambda.$$ Define a bilinear form $(\cdot, \cdot)_k$ on $\R[P]$ by $$( f, g )_k = (f \bar g \delta_k \bar \delta_k)_0,$$ which extracts the constant term (i.e.~the coefficient of $e^0 = 1$) in $f \bar g \delta_k \bar \delta_k$, where $\delta_k$ is the function defined in (\ref{eqn:delta-k-def}).  This bilinear form is symmetric and positive definite, and therefore defines an inner product on $\R[P]$.

For $\lambda \in P^+$, let $$M_\lambda = \frac{|W \cdot \lambda|}{|W|} \sum_{w \in W} e^{w(\lambda)}$$ be the monomial $W$-invariant (exponential) polynomial, and define $\mathrm{low}(\lambda)$ as the set of $\mu \in P^+$ such that $\lambda - \mu$ can be written as a linear combination of simple roots with non-negative integer coefficients.

\begin{definition} \label{def:HO-polys} \normalfont
For $\lambda \in P^+$, the {\it Heckman--Opdam polynomial} $P_{k, \lambda}$ is defined by
\begin{equation} \label{eqn:HO-poly1}
P_{k, \lambda} = \sum_{\mu \in \mathrm{low}(\lambda)} c_{\lambda \mu} M_\mu, \qquad c_{\lambda \lambda} = 1,
\end{equation}
and by the orthogonality relations
\begin{equation} \label{eqn:HO-orth}
(P_{k, \lambda} , \, M_\mu)_k = 0, \qquad \mu \in \mathrm{low}(\lambda), \ \mu \not = \lambda.
\end{equation}
Note that (\ref{eqn:HO-poly1}) and (\ref{eqn:HO-orth}) together determine the coefficients $c_{\lambda \mu} \in \R$ for $\mu \not = \lambda.$
\end{definition}

As $\lambda$ ranges over $P^+$ with $k$ fixed, the $P_{k,\lambda}$ form an $\R$-basis of the algebra $\R[P]^W$ of $W$-invariant elements in $\R[P]$.  Moreover this basis is orthogonal with respect to the inner product $(\cdot, \cdot)_k$.

\begin{example} \label{ex:HO-poly-examples} \normalfont
When all $k_\alpha$ are 0, $P_{0,\lambda} = M_\lambda.$

Let $\gog$ be a compact semisimple Lie algebra with root system $\Phi$, and identify $V$ with a Cartan subalgebra $\tot \subset \gog$.  Let $G$ be the connected, simply connected Lie group with $\mathrm{Lie}(G) = \gog$.  When $k = \vec 1$, $P_{\vec 1, \lambda}(ix) = \chi_\lambda(e^x)$ is the character of the irreducible $G$-representation with highest weight $\lambda$.

When $\Phi = A_{N-1}$, there is a natural injective homomorphism from $\R[P]^W$ to the usual ring of symmetric polynomials in $N$ variables.  This homomorphism is obtained by identifying $M_\lambda$ with the monomial symmetric polynomial $$m_\lambda(x_1, \hdots, x_N) = \frac{|S_N \cdot \lambda |}{N!} \sum_{\sigma \in S_N} \prod_{i=1}^N x_i^{\lambda_{\sigma(i)}}.$$ Under this identification, the Heckman--Opdam polynomials are Jack polynomials.  In particular, for $k = \vec 1$, they are Schur polynomials.  In light of Proposition \ref{prop:HO-majorization} below, Conjecture \ref{conj:HGF-majorization} therefore subsumes both 
the Schur polynomial inequality of \cite{CGS,Sra} and the classical Muirhead inequality for monomial symmetric polynomials \cite{Muirhead}.
\end{example}

Up to a normalizing factor, the Heckman--Opdam polynomials turn out to be specializations of the Heckman--Opdam hypergeometric function.  In particular, for $\lambda \in V$, define
\begin{equation} \label{eqn:tilde-c-def}
\tilde c(\lambda, k) = \prod_{\alpha \in \Phi^+} \frac{\Gamma(\langle \lambda, \alpha^\vee \rangle + \frac{1}{2} k_{\frac{1}{2} \alpha})}{\Gamma(\langle \lambda, \alpha^\vee \rangle + \frac{1}{2} k_{\frac{1}{2} \alpha} + k_\alpha)},
\end{equation}
where $k_{\frac{1}{2} \alpha} = 0$ if $\frac{1}{2} \alpha \not \in \Phi$.  Observe that if $\Phi$ is the root system of a compact Lie algebra $\gog$, we have $\tilde c(\lambda, \vec 1) = \Delta_\gog(\lambda)^{-1}.$  Set $$c(\lambda, k) = \frac{\tilde c(\lambda, k)}{\tilde c(\rho^{(k)}, k)}.$$
We then have the following relation between Heckman--Opdam polynomials and hypergeometric functions \cite[eq.~4.4.10]{HS}:
\begin{equation} \label{eqn:P-to-F}
F_{k,\lambda + \rho^{(k)}}(x) = c(\lambda + \rho^{(k)}, k) \, P_{k,\lambda}(x), \qquad x \in V, \ \lambda \in P^+.
\end{equation}
This relation generalizes the relation between the spherical functions $\phi^-_{-i(\lambda - \rho)}$ and $\phi^+_\lambda$ discussed in Section \ref{sec:cpt-sym}.  It immediately yields the following proposition.

\begin{proposition} \label{prop:HO-majorization}
Let $\lambda, \mu \in P^+$.  If Conjecture \ref{conj:HGF-majorization} holds, then the following are equivalent:
\begin{enumerate}
\item $\lambda \succeq \mu$,
\item $\tilde c(\lambda + \rho^{(k)}, k) P_{k,\lambda}(x) \ge \tilde c(\mu + \rho^{(k)}, k) P_{k,\mu}(x)$ for all $x \in V$.
\end{enumerate}
\end{proposition}

In other words, Conjecture \ref{conj:HGF-majorization} would imply a generalization of Proposition \ref{prop:compact-majorization} to the case of Heckman--Opdam polynomials.  A similar generalization of Proposition \ref{prop:euclidean-majorization} to generalized Bessel functions would also follow immediately via the rational limit in Example \ref{ex:gen-bessel}.

It is well known that Heckman--Opdam polynomials can be realized as a limit of Macdonald polynomials \cite{Mac}.  It is interesting to speculate about whether Conjecture \ref{conj:HGF-majorization}, if it holds, is itself a manifestation of an even more general monotonicity property of Macdonald polynomials with respect to $W$-majorization.

\section*{Acknowledgements}
Colin McSwiggen would like to thank Patrick McSwiggen for helpful discussions during the preparation of this manuscript.  The work of Colin McSwiggen is partially supported by the National Science Foundation under Grant No.~DMS 1714187.

The work of Jonathan Novak is partially supported by a Lattimer Fellowship, as well as by the Natural Science Foundation under Grant No.~DMS 1812288.

Both authors would like to thank the anonymous referee for helpful comments.

\bibliographystyle{alpha}

\end{document}